\documentclass[%
  twocolumn
]{mpi2015-cscpreprint}

\usepackage[american]{babel}

\usepackage{graphicx}
\usepackage{dsfont}
\usepackage{mathtools}
\usepackage{bm}
\usepackage{amssymb}
\usepackage{amsthm}

\newtheorem{theo}{Theorem}

\begin{document}
  

\title{A CFL-type Condition and Theoretical Insights for Discrete-Time Sparse Full-Order Model Inference}
  
\author[$1$,$\ast$]{Leonidas Gkimisis}
\author[$1$]{Süleyman Yıldız}
\author[$1, 2$]{Peter Benner}
\author[$2$]{Thomas Richter}

\affil[1]{Computational Methods in Systems and Control Theory (CSC), Max Planck Institute for Dynamics of Complex Technical Systems, Sandtorstra{\ss}e 1, 39106 Magdeburg, Germany}

\affil[2]{Institute for Analysis and Numerics, Otto-von-Guericke-Universit\"at Magdeburg, Universit\"atsplatz 2, 39106 Magdeburg, Germany}
  
\abstract{%
In this work, we investigate the data-driven inference of a discrete-time dynamical system via a sparse Full-Order Model (sFOM). We first formulate the involved Least Squares (LS) problem and discuss the need for regularization, indicating a connection between the typically employed $l_2$ regularization and the stability of the inferred discrete-time sFOM. We then provide theoretical insights considering the consistency and stability properties of the inferred numerical schemes that form the sFOM and exemplify them via illustrative, 1D test cases of linear diffusion and linear advection. For linear advection, we analytically derive a ``sampling CFL'' condition, which dictates a bound for the ratio of spatial and temporal discretization steps in the training data that ensures stability of the inferred sFOM. Finally, we investigate the sFOM inference for two nonlinear problems, namely a 2D Burgers' test case and the incompressible flow in an oscillating lid-driven cavity, and draw connections between the theoretical findings and the properties of the inferred, nonlinear sFOMs.}

\novelty{sparse FOM inference for dynamical systems in discrete time. Theoretical insights on the analytical solution  of the sparse FOM least-squares problem. Established connection between the stability of sparse FOM and the $l_2$ regularization of the least-squares problem.}

\maketitle

  
\section{Introduction}%
\label{sec:intro}

Physics-informed data-driven methods have emerged over the last few decades as a powerful tool for inference of system dynamics in physics and engineering. Leveraging physical knowledge in data-driven modeling enhances successful and robust predictions, while significantly restricting the solution space of the inference task. In this paper we focus on sFOM inference for discrete-time systems and present theoretical findings on consistency and stability, of infered models. We investigate these properties via linear and nonlinear numerical test cases.

The availability of high resolution numerical and experimental data due to the significant advancements in hardware and software, in parallel with the recent developments in linear algebra and machine learning methodologies, render data science as a key subject for tasks of system modeling, optimization and control. Data-driven methods have the potential to complement numerical solvers for such applications, resulting in significant time and resource savings.

The simultaneous usage of numerical simulations, experimental data and data-driven methods is central for the concept of digital twins \cite{Fuller2020,Kapteyn2022,Niederer2021}. The computational copy of a critical engineering system is continuously updated by the assimilation of numerical or experimental data, thus driving real-time control and design decisions. Such applications motivate the development of data-driven methods for dynamical systems, where the behaviour of the system is predicted using numerical or experimental data.

From the multitude of the available data-driven methods for dynamical systems \cite{Brunton2019}, we focus on approaches of physics-informed machine learning. These methods incorporate considerations on the underlying physics of dynamical systems and thus enable tailored data-driven solutions \cite{Karniadakis2021}. Developed approaches differ in terms of the assumed data availability and the degree to which system physics are a priori known \cite{peherstorfer2016data, raissi2017physics, brunton2016discovering, karachalios2022datadriven, SCHUMANN2023, Schaeffer2018, Mayo2007, SCHMID_2010, maddu2021stencilnet, Cenedese2022}. In the research direction of \cite{karachalios2022datadriven, Mayo2007}, derivation of state-space models is performed from input-output frequency response measurements. Assuming access to state measurements, \cite{SCHMID_2010, peherstorfer2016data} infer reduced-order models (ROMs) based on an a priori known physics-informed model structure, while \cite{brunton2016discovering, Schaeffer2018, Cenedese2022} derive the model structure from candidate analytical functions. In \cite{SCHUMANN2023, maddu2021stencilnet}, the inference of sparse models on the full-order level is studied through the utilization of a physics-informed model structure.

The latter method aims to infer numerical schemes with the use of state training data, which can then be used to construct sparse Full-Order Models (sFOMs) that can predict the system dynamics under different parameter values, initial or boundary conditions from those used during training \cite{SCHUMANN2023}. A potential advantage of this approach is the reduced dependence on the span of the training data snapshot matrix, compared to projection-based methods \cite{sfom_opinf}. In parallel, intrusive projection techniques, such as Proper Orthogonal Decomposition (POD), can be readily used to derive first-learn-then-project non-intrusive ROMs \cite{gkimisis2023adjacency, GkimisisFSI, prakashROM}. Nonetheless, inference on the full-order level inevitably introduces issues of stability \cite{prakash2024datadriven, SCHUMANN2023, sfom_opinf} and entails a significant offline and online computational cost \cite{sfom_opinf}.

In this study, we investigate the properties of sFOM inference for discrete-time systems and provide theoretical findings which highlight its potential and limitations. We first establish the connection between the $l_2$ regularization and the stability of the inferred sFOM, following the argumentation in \cite{sfom_opinf}. We derive a closed-form solution for the non-regularized LS problem and analyze its properties for 1D, linear diffusion. Performing a Taylor approximation of the closed-form solution, we also indicate a link between the training discretization in space and time, and the resulting sFOM stability, which we entitle ``sampling CFL'' condition. Similar sampling limitations have been discussed for other data-driven methods \cite{zhang2025}, while the trade-off between spatial and temporal discretization was hinted in \cite{pidmd} for sparse FOM inference. We validate this theoretical finding with numerical results for 1D linear advection. Finally, we test sFOM inference for two nonlinear test cases, namely a 2D Burgers' example and an oscillating-lid driven cavity example. The properties of the nonlinear sFOMs are extensively discussed, in connection with the presented theoretical findings.

The remainder of this paper is organized as follows;  In \Cref{sec:adj}, we introduce the sFOM inference problem formulation and present the employed regularization strategies. In \Cref{sec:theory}, we provide theoretical insights on the sFOM inference, which are validated with numerical results on 1D linear advection and diffusion examples. In \Cref{sec:num}, we showcase the capabilities and limitations of sFOM inference for two 2D nonlinear test cases and in \Cref{sec:conc}, we discuss potential future research directions and make concluding remarks.

\section{Adjacency-based Sparsity}
\label{sec:adj}

In this section, we formulate the basic problem linked to sparse FOM inference. The motivation behind this approach lies in the sparse structure of the discretized, differential operators, which is of spatially-localized nature for a wide range of physical systems \cite{pidmd}. Given a computational mesh and corresponding adjacency information, the sparse structure of the underlying PDEs is a priori known or can be approximated. In a data-driven context, exploiting adjacency-based sparsity allows to locally identify the involved operators \cite{SCHUMANN2023,maddu2021stencilnet, Pfaff2021}, thus enabling inference on the full-order level, via the storage and computation of a sparse FOM (sFOM).

\subsection{Problem Formulation}

We formulate the problem in line with the notation used in \cite{schum1st}. We focus on PDEs of the form

\begin{equation}
\label{contPDE}
    \frac{\partial u}{\partial t}=\mathcal{P}(u,u_x,u_y,u_z, u_{xx},u_{yy},u_{zz}...).
\end{equation}

\noindent where $\mathcal{P}(.)$ is a polynomial operator. We thus assume that $\mathcal{P}(.)$ consists of monomials of $u,u_x,u_y,u_z,$ $u_{xx},u_{yy},u_{zz}, \dots$. 
We examine $\mathcal{P}(.)$ after discretizing the continuous PDE~\eqref{contPDE} in space with a given numerical scheme that has local support. This results in a system of $n$ coupled ODEs. Using a one-step time discretization and focusing on some degree of freedom (DOF) $i$, we define the set of DOFs that correspond to the local support of $i$, as $Q_i$. These DOFs relate to spatially adjacent mesh nodes to the mesh node of $i$. The $i$-th ODE, discretized in time, can then be written as

\begin{equation}
\label{coefs_disc}
    u_i^{k+1}={\mathbf{f}(\mathbf{u}^k_{Q_i})^\top} \bm{\beta}_i,
\end{equation}

\noindent where $\mathbf{f}(\mathbf{u}^k_{Q_i})$ denotes the vector of geometrically adjacent variables to $i$ computed at timestep $k$ and $\bm{\beta}_i$ denotes the corresponding vector of numerical coefficients. An analogous formulation follows for continuous time; the l.h.s of the equation is the time derivative $\frac{\textrm{d}u}{\textrm{d}t}$ instead of $u_i^{k+1}$ and vectors $\mathbf{f}(\mathbf{u}^k_{Q_i})$, $\bm{\beta}_i$ are accordingly modified \cite{sfom_opinf}. In the following, we will focus on discrete-time systems in the form of \eqref{coefs_disc}.

In order to use \eqref{coefs_disc} in a data-driven framework, we should first specify the entries of $\mathbf{f}_{Q_i}$. For this reason, physics-based knowledge about the underlying PDE \eqref{contPDE} is required. As an example, if $\mathcal{P}(.)$ includes a quadratic nonlinearity, then after spatial discretization, the system of $n$ ODEs reads as 

\begin{equation}
\label{fom_structure}
\frac{\textrm{d}\mathbf{u}}{\textrm{d}t} =\mathbf{A} \mathbf{u} + \mathbf{H} ( \mathbf{u} \otimes \mathbf{u}) + \mathbf{c},
\end{equation}

\noindent where $\mathbf{A} \in \mathbb{R}^ {n \times n}$, $\mathbf{H}  \in \mathbb{R}^ {n \times n^2}$ and $\mathbf{c} \in \mathbb{R}^ {n}$. By discretizing \eqref{fom_structure} in time, we obtain the corresponding fully discrete system as

\begin{equation}
\label{fom_structure_disc}
\mathbf{u}^{k+1} =\mathbf{A}_d \mathbf{u}^k + \mathbf{H}_d ( \mathbf{u}^k \otimes \mathbf{u}^k) + \mathbf{c}_d.
\end{equation}

At this stage, we are interested to approximate $\mathbf{A}_d \in \mathbb{R}^ {n \times n}$ and $\mathbf{H}_d  \in \mathbb{R}^ {n \times n^2}$ in sparse form \cite{Bickel2012}. Substituting the $i$-th equation of \eqref{fom_structure_disc} into \eqref{coefs_disc} leads to

\begin{equation}
\label{fom_variables_i}
\mathbf{f}(\mathbf{u}^k_{Q_i}) = 
\begin{bmatrix}
  \mathbf{u}_{Q_i}^k \\ \mathbf{u}_{Q_i}^k \otimes \mathbf{u}_{Q_i}^k \\   1
\end{bmatrix},
\qquad
\bm{\beta}_i = 
\begin{bmatrix}
  {\mathbf{A}_d}_{i, Q_i}^\top \\
  {\mathbf{H}_d}_{i, E_i}^\top \\  \mathbf{c}_i
\end{bmatrix},
\end{equation}

\noindent with $\mathbf{f}(\mathbf{u}^k_{Q_i}), \; \bm{\beta}_i \in \mathbb{R}^{r_i^2+r_i+1}$, where $E_i$ corresponds to the $r_i^2$ combinations of DOFs in $Q_i$ \cite{GkimisisFSI}. We note that $ \bm{\beta}_i$ can be simplified to only consider the unique $r \choose 2$ of the total $r_i^2$ entries of the quadratic term. Knowing a physics-based structure as in \eqref{fom_structure_disc} then allows us to employ \eqref{coefs_disc} to infer the vector of coefficients $\bm{\beta}_i$, given snapshot data for $\mathbf{u}$ and a given set of geometrically adjacent DOFs, $Q_i$.

In a data-driven context, the inference of $\bm{\beta}_i$ via \eqref{coefs_disc} can be written as a least squares problem. Collecting state snapshot data for timesteps $t_0, \; \dots,\; t_N$, we construct the snapshot matrix $\mathcal{D}_i \in \mathbb{R}^{N \times r_i}$ and vector $\mathbf{d}_i \in \mathbb{R}^{N \times 1}$ as

\begin{equation}
\label{data_snapshots}
\mathcal{D}_i=\begin{bmatrix}
\mathbf{f}_{Q_i}^T(t_0) \\ \vdots \\ \mathbf{f}_{Q_i}^T(t_{N-1}) \end{bmatrix}, \qquad
\mathbf{d}_i=\begin{bmatrix}
u_{i}(t_1) \\ \vdots \\ u_{i}(t_N)
\end{bmatrix}.
\end{equation}

\noindent Then, the least squares (LS) problem for the numerical coefficients of DOF $i$ is written as

\begin{equation}
\label{LSform}
\min _{\bm{\beta}_i}{\left\|\mathcal{D}_i \bm{\beta}_i -  \mathbf{d}_i\right\|_{2}^2}.
\end{equation}

For the sFOM inference on uniform grids, \cite{sfom_opinf} suggested using data augmentation. Augmenting the system \eqref{LSform} with data from $p$, randomly selected DOFs in space enriches the data matrices' spatial and temporal information on the system dynamics. We denote the corresponding set of indices as $\alpha \in \mathbb{N}^p$. By such a data augmentation procedure, \eqref{LSform} is modified to

\begin{equation}
\label{daug}
\min _{\bm{\beta}_\alpha}{\left\|\mathcal{D}_\alpha \bm{\beta}_\alpha -  \mathbf{d}_\alpha \right\|_{2}^2},
\end{equation}

\noindent where 

\begin{equation}
\mathbf{d}_\alpha=\begin{bmatrix}
\mathbf{d}_{\alpha_1} \\ \vdots \\ \mathbf{d}_{\alpha_r}
\end{bmatrix}, \qquad
\mathcal{D}_a=\begin{bmatrix}
\mathcal{D}_{a_1} \\ \vdots \\ \mathcal{D}_{a_r}\end{bmatrix}.
\end{equation}

\subsection{Regularization strategies}
\label{regularization}

In practice, the solution of \eqref{LSform} is seldom applicable for inference of dynamical systems, due to numerical ill-conditioning of the LS problem and the resulting instability of the inferred dynamical system. Therefore, \eqref{LSform} is complemented with a regularization term which penalizes some norm of the solution $\bm{\beta}_i$. The authors in \cite{schum1st, SCHUMANN2023} performed a comprehensive study on the usage of both $l_2$ and sparsity-promoting $l_1$ regularization for \eqref{LSform}, also for noisy data. In \cite{prakash2024datadriven}, linear, stability-promoting constraints, motivated by the Gershgorin circle theorem \cite{gershgorin} were considered in addition to the $l_2$ regularization for the sFOM inference, leading to stable, continuous-time systems. In a similar direction, \cite{sfom_opinf} proposed a stability-promoting regularization with a closed-form solution for the inference of stable, continuous-time sFOMs. We focus on the LS formulation with $l_2$ regularization, which is written as

\begin{equation}
\label{LS_reg}
\min _{\bm{\beta}_i}{\left\|\mathcal{D}_i \bm{\beta}_i -  \mathbf{d}_i\right\|_{2}^2 + \eta_i \left\| \bm{\beta}_i \right\|_{2}^2},
\end{equation}

\noindent where $\eta_i$ is a hyper-parameter of the problem that weights the effect of the regularization term. The value of $\eta_i$ has been reported to be crucial for sFOM inference \cite{gkimisis2023adjacency, GkimisisFSI, schum1st}, as well as for other inference methods \cite{peherstorfer2016data, Schaeffer2018}. Indeed, study \cite{sfom_opinf} established a connection between the value of $\eta_i$ for the inference of continuous time systems and the corresponding radius of the $i$-th Gershgorin circle of the inferred linear operator.

For linear discrete-time systems, i.e. \eqref{fom_structure_disc} with $\mathbf{H}_d=\mathbf{0}, \; \mathbf{c}_d=\mathbf{0}$, the necessary and sufficient condition for stability is that the eigenvalues of $\mathbf{A}$ lie within the unit circle. This can be written as

\begin{equation}
\label{eigs_discrete}
|\lambda_i| \leq 1, \qquad \forall i \in \left[1, \dots, n\right]
\end{equation}

The connection between \eqref{eigs_discrete} and the values of the data-driven entries in \eqref{LS_reg} can be established analogously to \cite{prakash2024datadriven, sfom_opinf}, using the Gershgorin circle theorem \cite{gershgorin}. For discrete-time systems, a sufficient (but not necessary) stability condition for the inferred linear operator $\mathbf{A}$ is that all its Gershgorin circles lie within the unit circle. This can be written through \eqref{fom_variables_i}, for the inference of a linear system, as

\begin{equation}
\label{GerSt}
   \| \bm{\beta}_i \|_1 \leq 1, \qquad \forall i=\left[1, \dots, n\right],
\end{equation}

\noindent which is equivalent to the maximum norm stability of $\mathbf{A}$. The $l_1-l_2$ norm inequality for $\bm{\beta}_i$ is written as

\begin{equation}
\label{L1_L2}
  \| \bm{\beta}_i \|_2 \leq \| \bm{\beta}_i \|_1 \leq \sqrt{n_{Q_i}} \| \bm{\beta}_i \|_2.
\end{equation}

Observing the necessary stability condition \eqref{GerSt} and \eqref{L1_L2}, we then notice that the typically employed $l_2$ regularization in \eqref{LS_reg} is indeed stability-promoting for discrete-time systems, since it restricts the value of $\| \bm{\beta}_i \|_1$ in \eqref{GerSt}, by penalizing $\| \bm{\beta}_i \|_2$ (through \eqref{L1_L2}). Conversely, depending on the properties of the training data in \eqref{LS_reg}, \eqref{L1_L2} indicates that $ \bm{\beta}_i$ can a priori violate the sufficient stability guarantee in \eqref{GerSt}. We will further investigate such properties of the training data in the following.

\section{Theoretical Insights}
\label{sec:theory}
\subsection{Closed-form solution}

We first give an existence and uniqueness argument for the solution of \eqref{LSform}. Although we focus on the discrete-time formulation, the analysis for continuous-time systems follows analogously.

Let the simulation data be obtained from a discretized operator $\mathcal{P}(.)$, with a given numerical scheme and time discretization, resulting in \eqref{coefs_disc}.

\begin{theo}
If the selected $Q_i$ and $\mathbf{f}$ for the inference of $\bm{\beta}_i$ in \eqref{LSform} are the same as those employed for data collection via \eqref{coefs_disc} and if $\mathcal{D}_i$ in \eqref{LSform} has full column rank, there exists a unique solution  $\bm{\beta}_i$ to \eqref{LSform}. Furthermore, the solution $\bm{\beta}_i$ identically satisfies $ \mathbf{d}_i=\mathcal{D}_i \bm{\beta}_i$
\end{theo}

\begin{proof}
We assume that after discretization in time, the local support $Q_i$ and the physics-based model $\mathbf{f}$ on the right hand side of \eqref{LSform} is exactly the one used for data collection by numerical simulation of the discretized PDE via \eqref{coefs_disc}. Also, if $\mathcal{D}_i$ in \eqref{LSform} has full column rank, there exists a unique solution $\bm{\beta}_i$ to \eqref{LSform}. This solution $\bm{\beta}_i$ identically satisfies $ \mathbf{d}_i=\mathcal{D}_i \bm{\beta}_i$, by our first assumption. Solving \eqref{LSform} then results in identification of the numerical scheme used for the simulation.
\end{proof}

If the assumptions of the above theorem are met, the solution of \eqref{LSform} is trivial. However, these conditions are rarely known in practice; first, the assumed structure of the data-driven model, via $\mathbf{f}$, might differ from that of the numerical simulation and thus introduce modeling errors. Even when the mathematical structure of \eqref{coefs_disc} is exact, the numerical stencil $Q_i$ might differ from the one used during simulation. It should be noted that the sparsity pattern of the underlying operators is also affected by the employed time discretization scheme. Moreover, the matrix $\mathcal{D}_i$ is often rank deficient (e.g. for advection problems discussed in \cite{SCHUMANN2023}). Finally, the data in $\mathcal{D}_i$ and $\mathbf{d}_i$ might originate from an analytical PDE solution or from experiments. In such cases, the question of whether \eqref{LSform} admits a solution that can be used to infer the system dynamics is not trivial.

Starting from \eqref{LSform}, we write the SVD of $\mathcal{D}_i^T$ as

\begin{equation}
\label{data_svd}
\mathcal{D}_i^T=\Phi_i \Sigma_i \Psi_i^T    
\end{equation}

\noindent where $\Phi_i \in \mathbb{R}^{d \times d}$ is the spatial orthonormal basis of the data matrix, $\Sigma_i \in \mathbb{R}^{d \times d}$ is a diagonal matrix and $\Psi_i \in \mathbb{R}^{N \times d}$ is a temporal orthonormal basis. See \cite{morBauBH18} for a more in-depth interpretation of the basis properties of the left and right singular vectors on spatial and temporal domains.

Correspondingly, shifting the data timestep into the future, we obtain the SVD of $\mathbf{d}_i$ as

\begin{equation}
\label{lhs_svd}
\mathbf{d}_i^T=\begin{bmatrix}
u_{i}(t_1), \; \dots, \; u_{i}(t_{N}) \end{bmatrix}=\phi_{S_i} \sigma_{S_i} \bm{\psi}_{S_i}^T.
\end{equation}

\noindent Then, the solution of \eqref{LSform} based on the SVD can be explicitly written as

\begin{equation}
\label{soln}
\bm{\beta}_i= \Phi_i \Sigma_i^{\dagger} \Psi_i^T \bm{\psi}_{S_i} \sigma_{S_i} \phi_{S_i}^T.
\end{equation}

 We note that for an underdetermined system of equations, solution \eqref{soln} coincides with the LS minimum-norm solution \cite{strang2021introduction}. \eqref{soln} already exemplifies the need for the regularization in \eqref{LS_reg}; by taking the inverse of $\Sigma_i$, the contribution of singular vectors with a low energy content in the training data is amplified. We observe that the term $\Psi_i^T \bm{\psi}_{S_i}$ in \eqref{soln}, i.e., the correlation of the temporal orthonormal basis $\Psi_i$ with its shifted counterpart,  $\bm{\psi}_{S_i}$, is critical for the data-driven solution $\bm{\beta}_i$, as it is the term that encodes the evolution of the data in time. If $\bm{\psi}_{S_i}$ is in the span of matrix $\Psi_i$, such that  $\bm{\psi}_{S_i}= \Psi_i \mathbf{c}$, we obtain

\begin{equation}
\label{ill_diff}
\bm{\beta}_i= \Phi_i \mathbf{c}  \phi_{S_i}^T.
\end{equation}

\noindent due to orthogonality. The solution of \eqref{LSform} is then simply a linear combination of the spatial basis vectors in $\Phi_i$. 

\subsubsection{Example: Linear diffusion}
\label{subs:lindif}

The closed-form solution in \eqref{soln} is illustrated, by considering data given by

\begin{equation}
\label{diff_1}
    u(x,t)= \cos(x) e^{ -c t}.
\end{equation}

\noindent Data in the form of \eqref{diff_1} result from the analytical solution of the 1D linear diffusion problem

\begin{equation}
\label{diff_PDE}
    \frac{\partial{u}}{\partial{t}}=c  \frac{\partial^2{u}}{\partial{x}^2}, \qquad u(x,0)=\cos(x).
\end{equation}

\noindent Then, the data matrix \eqref{data_svd} is

\begin{equation}
\label{diff_svd}
\mathcal{D}_i = \begin{bmatrix}
    e^{ -c t_0} \\ \dots \\ e^{ -c t_{N-1}}
\end{bmatrix}\begin{bmatrix}
    \cos(x_{i-m}), \dots, \cos(x_i), \dots, \cos(x_{i+l})
\end{bmatrix},
\end{equation}

\noindent which comprises a rank-1 decomposition of $\mathcal{D}_i$. Similarly, \eqref{lhs_svd} is in this case
\begin{equation}
\label{diff_lhs_svd}
\mathbf{d}_i^T = \cos(x_i) e^{ -c \Delta t}
\begin{bmatrix}
    e^{ -c t_0}, \dots, e^{ -c t_{N-1}}
\end{bmatrix}.
\end{equation}

In the data driven context, we can solve the data-driven inference problem in \eqref{LSform}, assuming a linear system and a local support of $m+l+1$ indices, i.e., $Q_i=\{i-m, \dots, i, \dots, i+l\}$. This specific example falls under the special case where \eqref{ill_diff} is satisfied, since $\mathcal{D}_i$ is rank-1. We can analytically write out the minimum-norm solution to \eqref{soln} for this case as

\begin{equation}
\label{diff_soln}
\bm{\beta}_i = e^{-c \Delta t} \begin{bmatrix}
    \cos(x_{i-m}) \\ \dots \\ \cos(x_i) \\ \dots \\ \cos(x_{i+l})
\end{bmatrix} \cos(x_i) \left( \sum_{k=-m}^{l}{\cos^2(x_{i+k})} \right) ^{-1}.
\end{equation}

We observe that solution \eqref{diff_soln} highly depends on the sampling position $x_i$, while $\Delta t$ is only uniformly amplifying the solution $\bm{\beta}_i$.

By considering $m=l=1$ and fine space and time discretizations, such that $\Delta x \ll 1$ and $\Delta t \ll 1$, \eqref{diff_soln} simplifies to

\begin{equation}
\label{approx_diff}
\bm{\beta}_i \approx \frac{1-c \Delta t}{3 +2( \tan(x_i) \Delta x)^2} [1+ \tan(x_i) \Delta x  \; 1\; 1-\tan(x_i) \Delta x]^T.
\end{equation}

The numerical scheme in \eqref{approx_diff} is always stable, since the sufficient stability condition \eqref{GerSt} is satisfied for both $\| \tan(x) \| \leq \left( \Delta x \right)^{-1}$ and $\| \tan(x) \| \geq \left( \Delta x \right)^{-1}$. However, performing the Taylor expansion for $u_{i-1}$ and $u_{i+1}$ we see that the inferred scheme is of 0 order. From a different perspective, this can be viewed by examining the consistency of the numerical scheme (see Chapter 4.3, \cite{GeorgMay}), which is equivalent to the sum of numerical coefficients being equal to 1. This ensures that if the solution is constant across $Q_i$, then $u_i^{k+1}=u_i^k$. The sum of coefficients for \eqref{approx_diff} is

\begin{equation}
\label{diff_sum}
    \sum_{j=1}^{3}{\beta_i}_j = (1-c \Delta t) \frac{3}{3+2 (\tan(x_i) \Delta x)^2}.
\end{equation}

\noindent \eqref{diff_sum} indicates that the inferred numerical scheme is consistent only at the limit of $\Delta t, \Delta x \rightarrow 0$ and given that $\cos(x_i) \neq 0$. The convergence is linear with respect to $\Delta t$ and quadratic with respect to $\Delta x$. When $\cos(x_i) \rightarrow 0$, the solution of \eqref{LSform} results to stable, though potentially inaccurate numerical schemes. This highlights the effect of the spatial location of $i$ for the sFOM properties. Furthermore, it indicates a first limitation on generalizing the solution \eqref{soln} of the data-driven inference problem. Data augmentation can potentially alleviate this issue for uniform grids, by augmenting the matrices in \eqref{data_snapshots} with data from several, randomly selected DOFs. The generalization of the inferred numerical schemes is then possible \cite{sfom_opinf}.

We showcase the inferred linear sFOMs for \eqref{diff_1}, using either \eqref{LSform} or \eqref{daug}. For this test case, we use training data from the analytical solution \eqref{diff_soln} with $c=1$, $x \in [-\pi, \; \pi]$ and overall time of $T=10s$ for a range of spatial and temporal discretization values, $\Delta x \in [0.005, \;0.3]$ and $\Delta t \in [0.001,\;0.1]$. We proceed to solve both \eqref{LSform} and \eqref{daug} (without regularization) using a symmetric 3-point stencil ($Q_i=\{i-1,\;  i,\; i+1\}$). \Cref{fig:analyt_dif} illustrates the solution of different sFOMs at $T=10$ compared to the simulation data, as well as their average error $e$ with respect to the simulation data, defined as

\begin{equation}
\label{av_error}
e=\frac{\| \overline{u_{sFOM}(t=T)-u(t=T)} \|}{\max{(u(t=T)})}\times 100 \%,
\end{equation}

\noindent at the last timestep, $t=T$. Solving \eqref{LSform} without data augmentation, for each DOF $i$, leads to a stable sFOM (as expected by \eqref{approx_diff} and \eqref{GerSt}), exhibiting a linearly increasing error with respect to the training data time discretization, $\Delta t$. 

For the solution of \eqref{daug}, we append the matrices of one LS problem with data from $5\%$ of the total DOFs. The resulting $\bm{\beta}$ is then used for all $i \in [1, \;n]$ DOFs to form a linear sFOM. The stability properties of the resulting sFOM differ from those of \eqref{LSform}. For the same $\Delta t$, increasing the spatial discretization of the data $\Delta x$ can lead to unstable models. In contrast, \Cref{fig:analyt_dif} indicates that for a given $\Delta x=0.24$, the sFOM solution explodes when $\Delta t$ is smaller than a threshold value of $\Delta t=0.007$. In the following, we further focus on the relation between the stability of an inferred sFOM and the discretization properties of the training data.

\begin{figure}[!htb]
    \centering
    \includegraphics[clip,width=1.1\columnwidth]{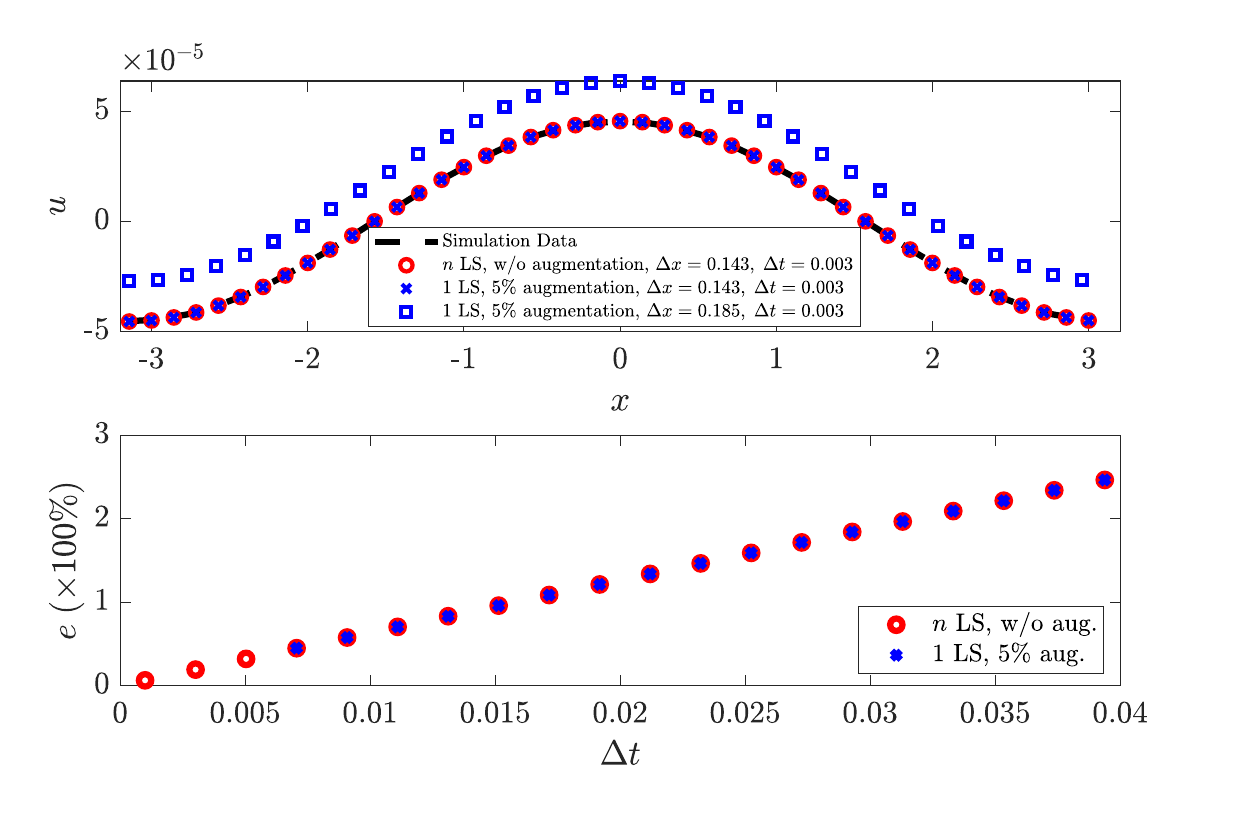}
    \caption{Numerical results for the inference of linear diffusion dynamics \eqref{diff_soln} via \eqref{LSform} and \eqref{daug}. \textit{Top}: Both \eqref{LSform} and \eqref{daug} yield accurate results for the state prediction, though training data discretization in space ($\Delta x$) and time ($\Delta t$) affect the accuracy of the augmented data solution from \eqref{daug}. \textit{Bottom}: Linear increase in the prediction average error $e$ at $t=10$, for different training data time discretizations $\Delta t$ and $\Delta x=0.24$. For very low $\Delta t$ values, \eqref{daug} yields unstable results.}
    \label{fig:analyt_dif}
\end{figure}

\subsection{A sampling CFL condition for linear, 1D problems}
\label{CFL_section}

We further analyze \eqref{LSform} to establish a connection between the stability of the inferred sFOM and the employed sampling in space and time for the collected snapshot data in \eqref{data_snapshots}. For this task we limit ourselves to linear, 1D problems. Then, for a sufficiently fine discretization in space and time ($\Delta x$ and $\Delta t$), we can derive a Taylor approximation of \eqref{soln}. 

For a one-dimensional linear system, we assume a generally unsymmetric stencil, $Q_i=\{i-m,\;...,\;i,\;...,\;i+l\}$. We then perform a first-order (or rank-2) Taylor approximation of the entries of $\mathcal{D}_i$ in \eqref{data_snapshots} as

\begin{multline}
\label{taylor_data}
\mathbf{f}(\mathbf{u}^k_{Q_i})=\begin{bmatrix}
u_{i-m}^k & ... & u_i^k & ... & u_{i+l}^k
\end{bmatrix} \\
\approx \begin{bmatrix}
u_i^k & \frac{\textrm{d}u}{\textrm{d}x}|_i^k
\end{bmatrix}\begin{bmatrix}
1 & ... & 1 & ... & 1\\
-m\Delta x & ... & 0 & ... & l\Delta x
\end{bmatrix}.
\end{multline}

\noindent We note that the 1st order Taylor expansion in \eqref{taylor_data} corresponds to using a two-dimensional, generally non-orthogonal basis to approximate $\mathcal{D}_i$ in \eqref{data_snapshots}. We similarly express the entries of the right-hand side vector $\mathbf{d}_i$ in \eqref{data_snapshots} as

\begin{equation}
\label{taylor_future}
u_{i}^{k+1} \approx \begin{bmatrix}
u^k_i & \frac{\textrm{d}u}{\textrm{d}t}|^k_i \end{bmatrix} \begin{bmatrix}
1 \\ \Delta t
\end{bmatrix}.
\end{equation}

This approximation can be viewed as an explicit Euler approximation of the time derivative $\frac{\textrm{d}u}{\textrm{d}t}|^k_i$. By substituting \eqref{taylor_data} and \eqref{taylor_future} into \eqref{data_snapshots}, we can derive a first-order approximation for the solution of \eqref{LSform}. To do so, we first write out an $(n_t-1)\times 2$ overdetermined system with vector of unknowns $\mathbf{y} \in \mathbb{R}^2$

\begin{equation}
\label{eqn:truncated-full}
\begin{bmatrix}
u_i & \frac{\textrm{d}u}{dx}|_i
\end{bmatrix} \mathbf{y}=\mathbf{d}_i,
\end{equation}

\noindent where 
\begin{equation}
   \label{eqn:truncated-y}
   \mathbf{y}= \begin{bmatrix}
1 & ... & 1 & ... & 1\\
-m\Delta x & ... & 0 & ... & l\Delta x
\end{bmatrix} \bm{\beta}_i.
\end{equation}

After solving this system with respect to $\mathbf{y}$, we take the minimum-norm solution of the resulting underdetermined $2 \times (l+k+1)$ system \cite{strang2021introduction} with respect to $\bm{\beta}_i$. Denoting 
\begin{equation}
   \mathbf{A}= \begin{bmatrix}
1 & ... & 1 & ... & 1\\
-m\Delta x & ... & 0 & ... & l\Delta x
\end{bmatrix},
\end{equation}

\noindent the minimum-norm solution of the truncated problem \Cref{eqn:truncated-y} is written as
\begin{equation}
    \label{eqn:lst-norm}
    \bm{\beta}_i=\mathbf{A}^T(\mathbf{A}\mathbf{A}^T)^{-1}\mathbf{y}.
\end{equation}

\noindent Since $\mathbf{A}\mathbf{A}^T\in \mathbb{R}^{2\times 2}$ has full rank, it is invertible. Indexing the entries of $\bm{\beta}_i$ with $j$, we write the analytical solution of \eqref{eqn:truncated-full} as

\begin{multline}
\label{approxD}
{\bm{\beta}_i}_j = \left((ab-d^2)(c_1 c_2 - c_3^2)\right)^{-1} \left( c_2 \left( ab-d^2 \right)+ \right. \\ 
\left. c_2 \Delta t \left(b g -d e\right)  - \frac{ c_3  \Delta t}{ \Delta x} \left(a e - g d\right)  + j \left( -c_3 \left(a b - d^2\right) + \right. \right. \\ \left. \left. - \Delta t c_3 \left( b g -d e\right) + \frac{c_1 \Delta t}{\Delta x}\left( a e - g d\right) \right)\right) ,
\end{multline}

\noindent where
\begin{multline}
\label{coefs}
\left\{\begin{array}{ll}
\begin{aligned}
 c_1&=m+l+1 \\ c_2&= 1/6\left[m(m+1)(2m+1)+l(l+1)(2l+1)\right] \\ c_3&=1/2\left[(l-m)(m+l+1)\right]
 \end{aligned}
\end{array}\right.,
\end{multline}

\noindent and coefficients $a,\;b,\;d,\;g,\;e$ are depending on the system at hand and the DOF $i$:

\begin{multline} 
\label{extracoefs}
a=\mathbf{u}_i^T \mathbf{u}_i, \qquad  b=\left.\frac{\textrm{d}\mathbf{u}}{dx}\right\vert_i^T \left.\frac{\textrm{d}\mathbf{u}}{dx}\right\vert_i, \qquad d=\mathbf{u}_i^T \left.\frac{\textrm{d}\mathbf{u}}{dx}\right\vert_i, \qquad \\ g=\mathbf{u}_i^T \left.\frac{\textrm{d}\mathbf{u}}{\textrm{d}t}\right\vert_i, \qquad e=\left.\frac{\textrm{d}\mathbf{u}}{dx}\right\vert_i^T \left.\frac{\textrm{d}\mathbf{u}}{\textrm{d}t}\right\vert_i. \end{multline}

The term $(ab-d^2)$ in the denominator of \eqref{approxD} is always positive due to the Cauchy-Schwarz inequality, while after algebraic manipulation, $(c_1 c_2 - c_3^2)$ can also be shown to be positive. We observe that the entries of \eqref{approxD} are linearly dependent on the adjacent DOF index $j$, such that \eqref{approxD} can be written as

\begin{equation}
  \label{approx2}
{\bm{\beta}_i}_j=\frac{K_{1_i}+j \;K_{2_i}}{(ab-d^2)(c_1 c_2 - c_3^2)},
\end{equation}

\noindent with

\begin{equation} 
\label{anal_approx}
\left\{\begin{array}{ll}
\begin{aligned}
K_{1_i} &=  c_2 \left( ab-d^2 \right)+ c_2 \Delta t \left(b g -d e\right)  - \frac{c_3 \Delta t}{ \Delta x} \left(a e - g d\right)
\\
K_{2_i} &=  -c_3 \left(a b - d^2\right) - \Delta t c_3 \left( b g -d e\right) + \frac{c_1 \Delta t}{\Delta x}\left( a e - g d\right).
\end{aligned}
\end{array}\right.
\end{equation}

Due to the 1st order Taylor expansion used, approximation \eqref{approx2} can provide insight mainly on the inference of systems with first-order spatial derivatives, such as linear advection. The preceding analysis can be extended to higher-order approximations, two-dimensional or/and nonlinear systems, however, analytical computations become increasingly cumbersome.

We are interested in the stability of the numerical scheme in \eqref{approxD}, or conversely the stability of the resulting, inferred linear sFOM. By substituting \eqref{approxD} into the sufficient stability condition in \eqref{GerSt}, we obtain a sufficient stability condition for the resulting linear system as

\begin{equation}
  \label{stab_theory}
\sum_{j=-m}^l\|{K_{1_i}+j \;K_{2_i}}\|\leq{(ab-d^2)(c_1 c_2 - c_3^2)}, \; \forall i \in [1, \dots, n].
\end{equation}

The sufficient stability condition \eqref{stab_theory} indicates that the stability of the inferred sFOM depends on the width of the selected numerical scheme via the coefficients in \eqref{coefs}, the properties of the solution at DOF $i$ via the coefficients in \eqref{extracoefs}, but also the data sampling in space and time,  $\Delta x, \; \Delta t$ due to \eqref{anal_approx}. This is introducing an important consideration for data sampling, which was hinted in previous works \cite{pidmd,SCHUMANN2023}. The training data time discretization $\Delta t$ and space discretization $\Delta x$ can affect the stability of the derived data-driven sFOM. The first-order Taylor approximation in \eqref{taylor_data} results to a sampling CFL condition for the data-driven model.

This limitation for the data-driven inference of any physical problem which exhibits transport phenomena can be intuitively grasped analogously to the numerical CFL condition. For a given stencil $Q_i$, the timestep $\Delta t$ can be chosen sufficiently large, such that the transport phenomena are not captured by the data in $\mathcal{D}_i$. This is schematically illustrated in \Cref{fig:CFL}. It is then not possible to use the solution of \eqref{LSform} to successfully infer the system dynamics. This intuitive explanation suggests that any additional regularization to \eqref{LSform} would not alleviate this issue.

\begin{figure}[!htb]
    \centering
    \includegraphics[clip,width=\columnwidth]{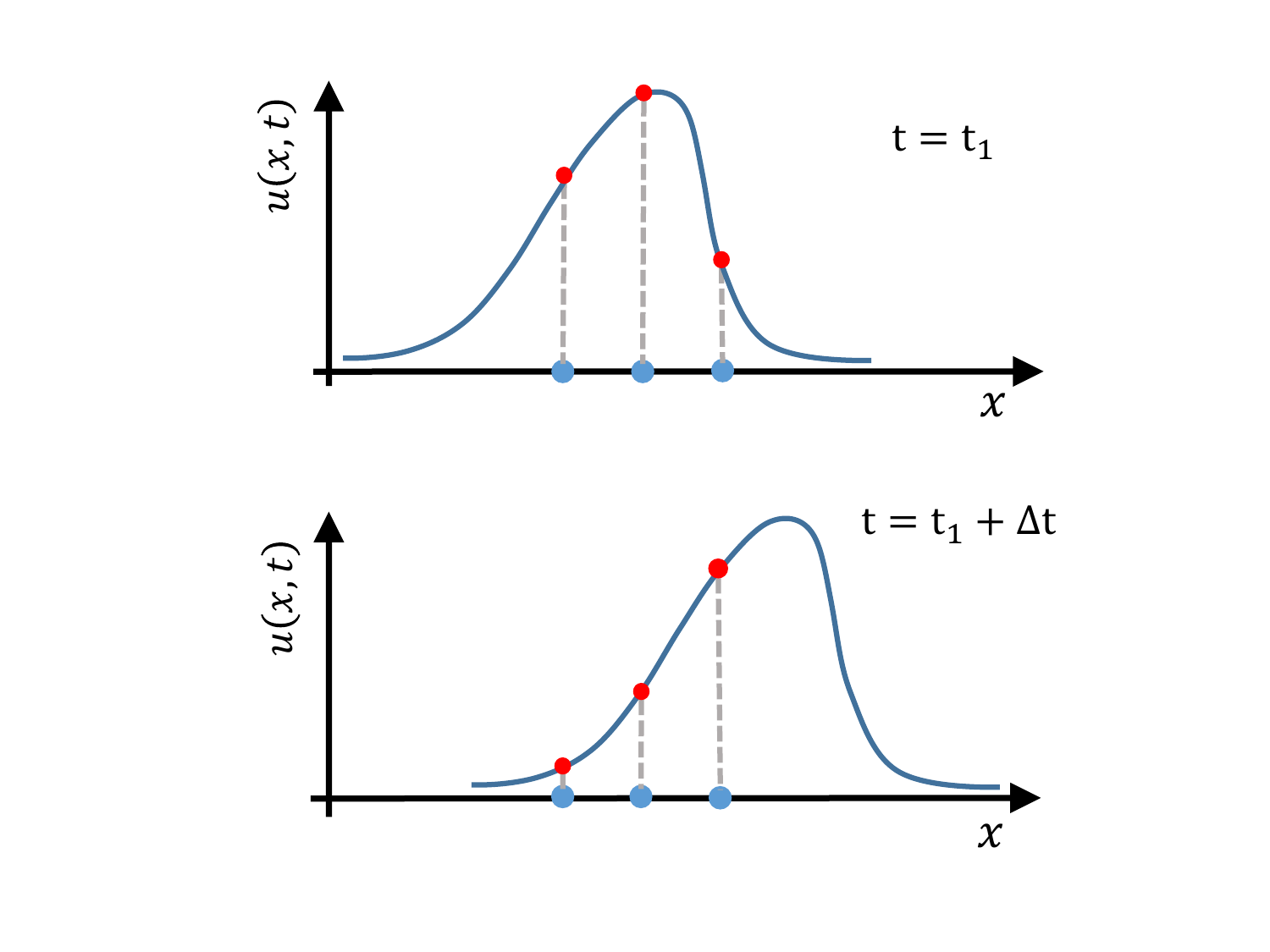}
    \caption{Schematic representation of sampling CFL condition: If the timestep $\Delta t$ is chosen large enough, the system dynamics cannot be captured by the red measurements in the stencil $Q_i$.}
    \label{fig:CFL}
\end{figure}

\subsubsection{Example: Linear Advection}

To exemplify the theoretical insights in \Cref{CFL_section}, we test a case of 1D, linear advection, which is modeled by

\begin{equation}
\label{linadv1d}
    \frac{\partial u}{\partial t} =c \;\frac{\partial u}{\partial x}.
\end{equation}

\noindent We infer a linear sFOM (i.e. \eqref{fom_structure_disc} with $\mathbf{H}_d=\mathbf{0}, \; \mathbf{c}_d=\mathbf{0}$), by collecting data from the analytical solution of \eqref{linadv1d}, given arbitrary initial conditions. We consider symmetric numerical stencils ($m=l$) and derive the 1st-order approximation of the data-driven sFOM via \eqref{approxD}. Substituting  \eqref{approxD} for \eqref{linadv1d} into the sufficient stability condition \eqref{stab_theory}, we obtain an explicit formula for the ``sampling CFL" condition

\begin{equation}
\label{CFL_linadv}
   \frac{\Delta t}{\Delta x} \leq \frac{m+1}{3c}.
\end{equation}

This inequality provides a priori bounds for the spatial and temporal discretization of the collected data via $\Delta x, \Delta t$ that guarantee stability of the data-driven, inferred sFOM. It should be noted that conditions like \eqref{CFL_linadv} are similarly expected for sFOM inference in continuous time, depending on the employed 
numerical scheme for the time derivatives computation \cite{sfom_opinf}.

We test the above bound numerically on a domain $x \in [-1,\; 1]$, for $c=1$. We obtain the training data by discretizing the analytical solution of \eqref{linadv1d} for $u(t=0)=\cos(x)$ with $\Delta x \in [0.005,\; 0.01]$ and $\Delta t \in [0.002, \; 0.02]$, for a total time $T=5$.

The numerical results are presented in \Cref{fig:linadv_CFL}. We solve \eqref{LSform} for all DOFs, considering a 3- and a 5-point stencil ($m=1,\; m=2$), for the aforementioned range of spatial and temporal discretizations $\Delta x$ and $\Delta t$ of the training data. We examine the eigenvalues of the inferred linear operator and record the average simulation error \eqref{av_error} of the resulting sFOM at $T=5$. We observe that the pairs of $\Delta x$ and $\Delta t$ beyond which the inferred sFOMs are unstable \eqref{eigs_discrete}, closely follow the estimate in \eqref{CFL_linadv}. The inferred dynamics remain accurate even slightly above this threshold, for the finite time of $T=5$.

\begin{figure}[!htbp]
    \centering
    \includegraphics[clip,width=1.05\columnwidth]{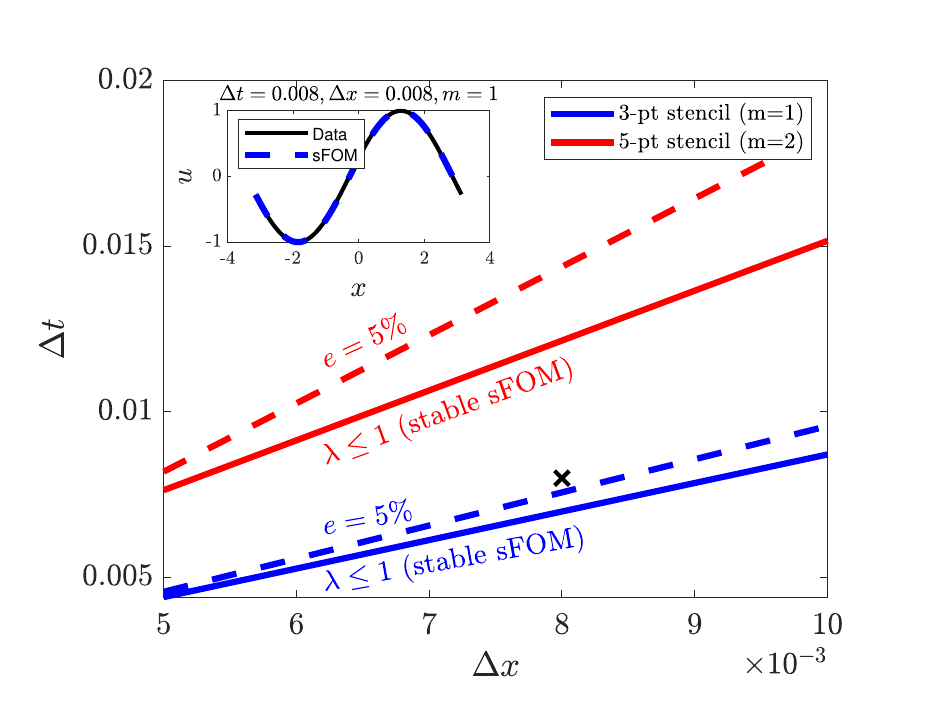}
    \caption{Discrete-time sFOM for linear advection: For a given numerical stencil, the stability of the inferred sFOM depends on the discretization of the training data $\Delta x, \Delta 
 t$, following the sampling CFL condition in \eqref{CFL_linadv}. The 3-pt stencil sFOM prediction for data with $\{\Delta x = 0.008, \; \Delta t = 0.008 \}$ (marked with a black cross) is given on the upper right corner of the figure: The stability condition \eqref{CFL_linadv} is slightly violated, but the results in finite time remain accurate.}
    \label{fig:linadv_CFL}
\end{figure}

\section{Nonlinear test cases}
\label{sec:num}

We further investigate the task of sFOM inference for two nonlinear test cases, namely a 2D Burgers' equation benchmark and the incompressible flow in an oscillating lid-driven cavity. For this task, we exploit the theoretical insights obtained in \Cref{sec:theory} and comment on potential limitations of sFOM inference.

\subsection{2D Burgers' equation}
\label{Burgers}

The single-variable, two-dimensional viscous Burgers' equation writes as

\begin{equation}
\label{burgers_pde}
\frac{\partial u}{\partial t}=c \; u \left(\nabla . u \right)+\nu \; \nabla^2 u,
\end{equation}

\noindent where $c$ is the transport coefficient and $\nu$ is the kinematic viscosity. We consider periodic boundary conditions on a domain $(x,y) \in [0,1]\times[0,1]$ and a parameterized initial condition

 \begin{equation}
 \label{ini_burg}
u(t=0)=\alpha \;e^{-\mu{(x-0.5)}^2} e^{-\mu{(y-0.5)}^2}.
\end{equation}

\noindent To obtain the numerical solution for this system for $c=0.1$, $\nu= 10^{-3}$, $\alpha=1$ and $\mu=10$, we use a second-order, central-difference scheme for diffusion and a second-order, forward-difference scheme for advection. The training data are obtained by integrating \eqref{burgers_pde} with the MATLAB function \texttt{ode45}, using $\Delta t=0.01 \;s$ and $\Delta x=0.02\;m$, for a total time $T=10 \;s$.

In a discrete-time, data-driven setting, we aim to infer a quadratic 
model (due to the quadratic nonlinearity in \eqref{burgers_pde}), which can be formulated as

\begin{equation}
\label{2Dburg}
\mathbf{u}^{k+1}= \mathbf{A} \mathbf{u}^{k}+\left(\mathbf{H} {\mathbf{u}^{k}} \right) \circ {\mathbf{u}^{k}},
\end{equation}

\noindent where $\circ$ signifies the element-wise product of two vectors and $\mathbf{A}, \mathbf{H} \in \mathbb{R}^{n \times n}$. We use $\left(\mathbf{H} {\mathbf{u}^{k}} \right) \circ {\mathbf{u}^{k}}$ as a special case of the general nonlinear term $\mathbf{H} \left( \mathbf{u}^{k} \otimes \mathbf{u}^{k}\right)$ \cite{peherstorfer2016data}, to closely follow the structure of the quadratic nonlinearity in \eqref{burgers_pde}.

We employ \eqref{LS_reg} with the use of data augmentation \eqref{daug} to infer an sFOM, based on structure \eqref{2Dburg}.
We solve one least-squares problem for a $5 \times5$ stencil, with $m=500$ augmentation points and an $l_2$ regularization term. $30$ logarithmically spaced regularization values $\eta$ (see \eqref{LS_reg}) are tested in the range $\eta \in [ 10^{-5}, \; 10^0 ]$. Due to the different scaling of the linear and quadratic terms \cite{mcquarrie2021data}, the quadratic operator entries of $\bm{\beta}_i$ (see \eqref{fom_variables_i}) are regularized by a factor $g \times \eta$, where $g \in [10,20,50,80,100]$. The optimal values for $g$ and $\eta$ in the aforementioned ranges are computed via an L-curve \cite{Hansen2000, gkimisis2023adjacency}. The inferred sFOM is assembled by assigning the optimal solution of the regularized \eqref{daug} to all rows of operators $\mathbf{A}, \mathbf{H}$ in \eqref{2Dburg} (due to the prescribed periodic boundary conditions).

Both the numerical solution and the sFOM prediction at the end of the simulation time $T=10\; s$ are given in \Cref{fig:2Dbrg}. The selection of the $\eta$ and $g$ values were found to be crucial for the quality of the sFOM. A high $\eta$ value leads to stable, though inaccurate results, as expected by \eqref{GerSt}, while a very low $\eta$ value leads to unstable eigenvalues of $\mathbf{A}$ in \eqref{2Dburg}.

\begin{figure}[!htb]
    \centering
\includegraphics[clip,width=1.1\columnwidth]{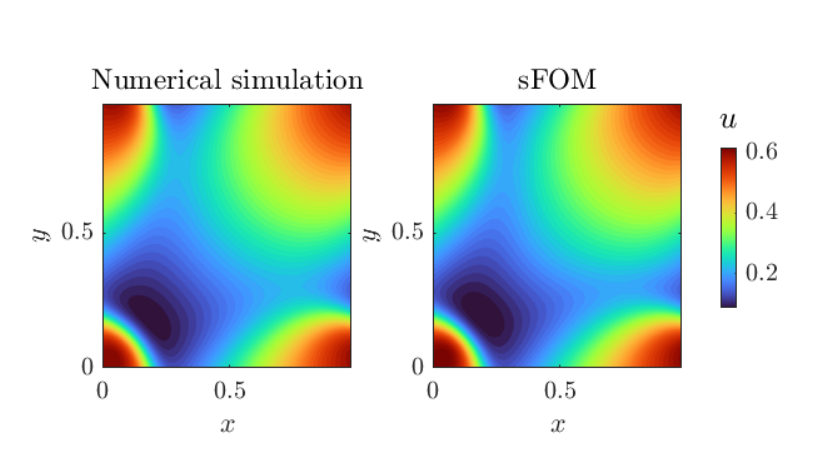}
    \caption{sFOM predictions for 2D Burgers' test case, at time $T=10\;s$, for $a=1$ and $\mu=10$: Successful reproduction of the system dynamics in the training data, with the solution of only 1 LS problem, while all eigenvalues of $\mathbf{A}$ lie within the unit circle.}
    \label{fig:2Dbrg}
\end{figure}

We simulate the obtained sFOM for initial conditions different from those used for training. In particular, we vary the parameters in \eqref{ini_burg}, with $\alpha \in [0.1,0.2,0.5,1,2,5]$ and $\mu \in [1,2,5,7,10,15,20]$. Parameter $\alpha$ corresponds to the scaling of the initial condition, which in turn scales with the advection speed of the system, while $\mu$ dictates the width of the initial Gaussian in \eqref{ini_burg}. The average relative error $e$ of the sFOM with respect to the numerical solution of \eqref{burgers_pde} at the final time $T=10$ for different $\alpha$ and $\mu$ values is plotted in \Cref{fig:2Dbrg_ic}. For values $\alpha \leq 1$, the sFOM provides accurate predictions, especially for low values of $\mu$. For $\alpha \geq 1$, the sFOM simulation is unstable for almost all values of $\mu$. This can be attributed to the system's increased advection speed, as well as the increased norm of the quadratic term \cite{Benner2015}. A potential solution to this issue could be to perform sFOM inference using training data for different initial conditions.

\begin{figure}[!htb]
    \centering
    \includegraphics[clip,width=\columnwidth]{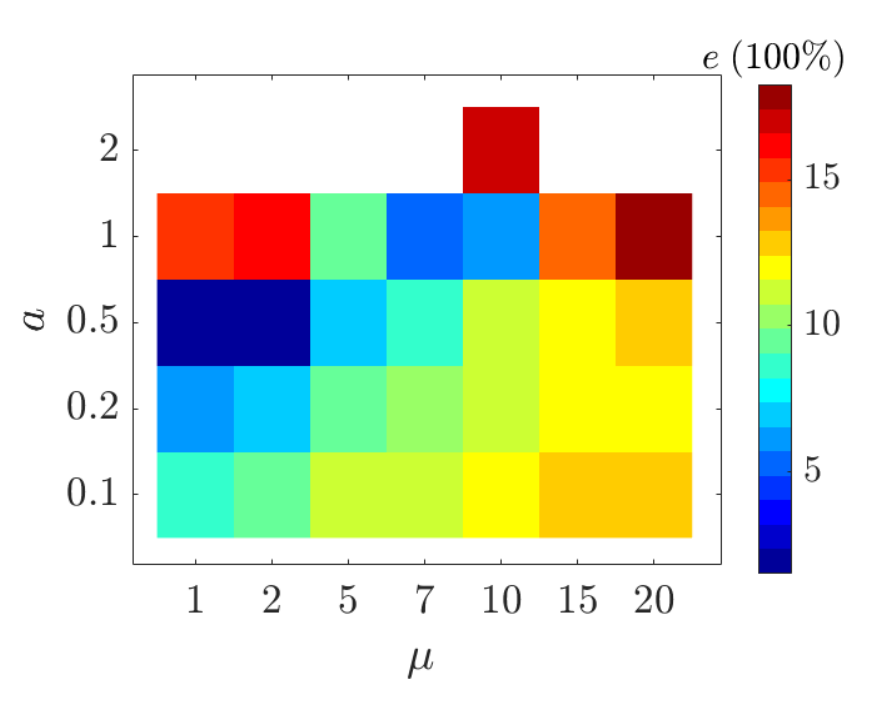}
    \caption{Average, relative error of the inferred sFOM \eqref{2Dburg} at $T=10 \; s$ for different initial conditions: The sFOM model makes accurate predictions for values of $\alpha \leq 1$ and small $\mu$ in \eqref{ini_burg}. The average, relative error for $\alpha=1$, $\mu=10$ corresponds to the state prediction in \Cref{fig:2Dbrg}.}
    \label{fig:2Dbrg_ic}
\end{figure}

\subsection{Lid-Driven Cavity}
 Increasing modeling and mathematical complexity, we investigate the performance of sFOM to an incompressible fluid dynamics test case. Past works have considered the inference of dynamical systems for fluid-structure interactions \cite{gkimisis2023adjacency, GkimisisFSI} or linear 2D test cases \cite{prakashROM} by projecting the inferred sFOM to a reduced basis. We hereby focus more on the properties of the sFOM, without projection.

The incompressible Navier-Stokes equations for the lid-driven cavity test case are formulated as

\begin{equation} 
\label{NS_incomp}
\left\{\begin{array}{ll}
\begin{aligned}
\frac{\partial \mathbf{u}}{\partial t} + \left({\mathbf{u}  \cdot \nabla} \right) \; {\mathbf{u}}
-\frac{1}{\rho} \nabla \cdot  {\boldsymbol{\sigma}}\left({\mathbf{u}}, {p}\right) & = \vec g, \quad\\ 
\nabla \cdot {\mathbf{u}} & =0, \quad
\end{aligned}
\end{array}\right. \mathbf{x}\in [0, L]\times[0, L],
\end{equation}

\noindent where ${\mathbf{u}}$ is the velocity field, ${p}$ is the pressure, ${\boldsymbol{\sigma}}$ is the stress tensor given by
\begin{equation}
\label{sigF}
 {\boldsymbol{\sigma}}\left({\mathbf{u}}, {p}\right) =-{p}\mathbb{I}+\rho_f \nu_f \left( {\nabla} {\mathbf{u}} + {\nabla} {\mathbf{u}}^T \right)
\end{equation}

\noindent for a Newtonian fluid. The cavity lid has a prescribed velocity $u_{top}$ along $x$, while the no-slip and no-through conditions ($\mathbf{u}=\mathbf{0}$) are assigned to the side and bottom walls of the lid \cite{ABDELMIGID2017123}.

We set an oscillating lid velocity over time $u_{top}=2 \cos(\pi t)$ and simulate \eqref{sigF} using Chorin's operator splitting method \cite{Chorin1968}, over a total time of $T=4\;s$ with $\Delta t =0.002$, $L=1$ and $\Delta x = 0.026$. 

After discretizing \eqref{NS_incomp} in space and time and canceling out the pressure term, the resulting system has a quadratic nonlinearity \cite{Benner2021} and can be expressed as

\begin{equation}
\label{lid_disc}
    \mathbf{u}^{k+1}= \mathbf{A}  \mathbf{u}^k + \mathbf{H}\left( \mathbf{u}^k \otimes \mathbf{u}^k \right) + \mathbf{c}.
\end{equation}

\noindent with $\mathbf{A} \in \mathbb{R}^{n\times n}, \mathbf{H} \in \mathbb{R}^{n\times n^2}, \mathbf{c}\in \mathbb{R}^{n}$. We approximate operators $\mathbf{A}, \mathbf{H}$ to be banded and sparse \cite{gkimisis2023adjacency, Bickel2012} and infer them via a sFOM \eqref{LS_reg}. We select a 3rd-order adjacency local support $Q_i$ for the linear operator and a 1st-order adjacency local support for the quadratic operator. We employ an $l_2$ regularization, optimized via an L-curve, for $30$ logarithmically spaced $\eta$ values in the range $\eta \in [ 10^{-3}, \; 10^0 ]$. The entries of $\mathbf{H}$ are regularized by a factor of $100 \times \eta$. The sFOM was inferred for the first $2 \;s$ of simulation data, while $t \in [2,\; 4] \;s$ comprises the testing time. 

\Cref{fig:error_lid} illustrates the flowfield average error over time for both components of $\mathbf{u}$. We observe that the error reaches a maximum value of $\approx 10\%$ for the vertical velocity component during testing time. However, the inferred sFOM produces accurate predictions of the system dynamics beyond the training regime. This is supported by the sFOM state prediction in \Cref{fig:2Dlid}, where the main circulation patterns at $t=4 \;s$ are accurately predicted.

\begin{figure}[!htb]
    \centering
\includegraphics[clip,width=\columnwidth]{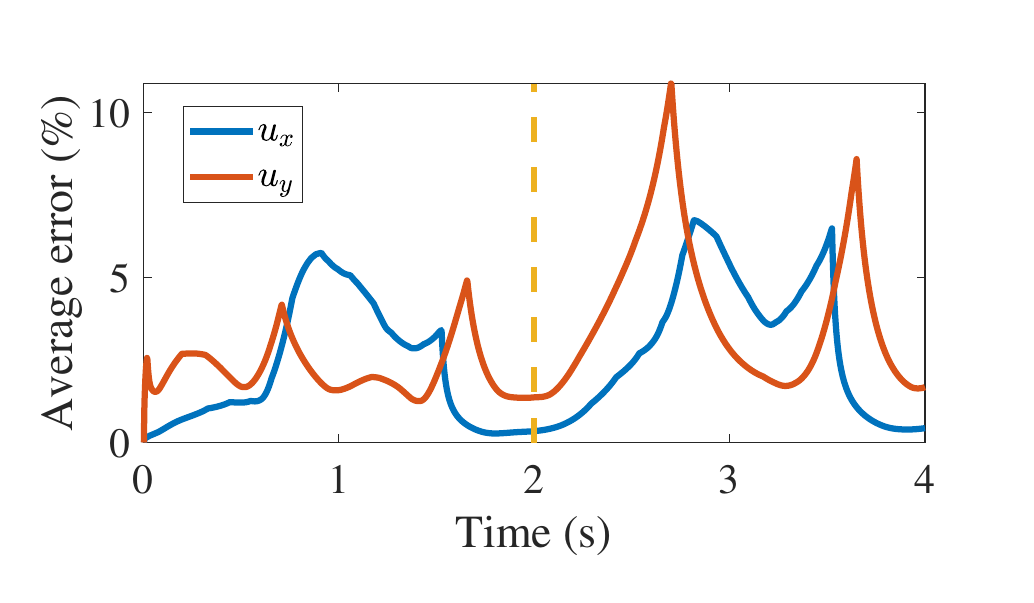}
    \caption{ Average flowfield sFOM error over time for oscillating lid-driven cavity: Accurate predictions by the quadratic sFOM \eqref{lid_disc} with 1600 DOFs beyond the training time of $2 \;s$.}
    \label{fig:error_lid}
\end{figure}

\begin{figure}[!htb]
    \centering
\includegraphics[clip,width=1.10\columnwidth]{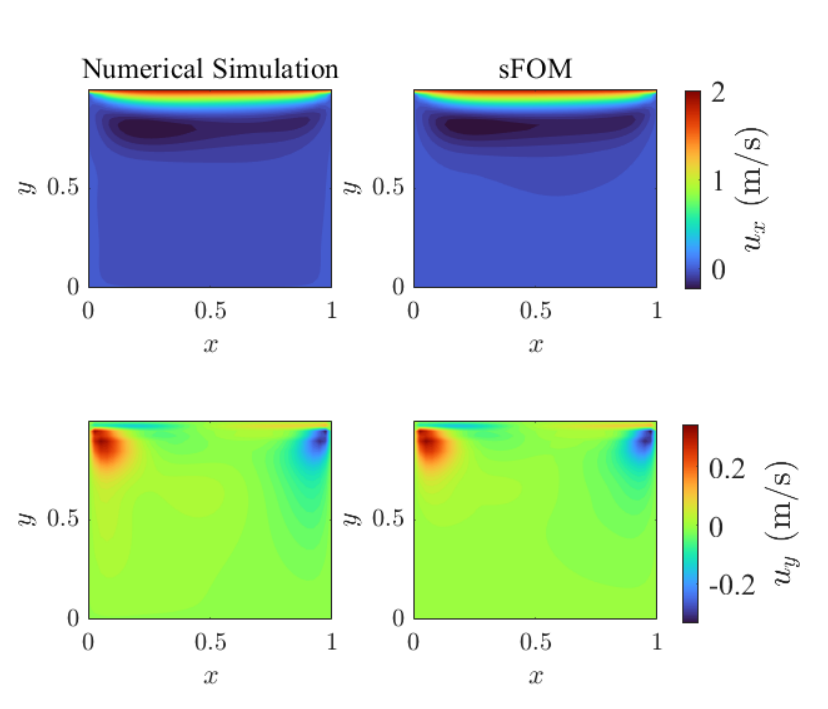}
    \caption{sFOM state predictions for oscillating lid-driven cavity at $t=4\;s$: The inferred sFOM accurately predicts the main features of the flow for both $u_x$ and $u_y$.}
    \label{fig:2Dlid}
\end{figure}

We should mention that using an extended local support $Q_i$ with the 3rd-order adjacent DOFs for the linear operator was found necessary for successful sFOM inference. This could be attributed to a stability limitation, as the one showcased in \eqref{CFL_linadv}. We also stress that the operators in \eqref{lid_disc} are only approximated as banded and sparse, which introduces an additional modeling error to the inferred sFOM. Finally, although the employed setup led to accurate predictions for this test case of incompressible fluid dynamics, it does not guarantee the algebraic condition of mass conservation \eqref{NS_incomp}. As a future research endeavour, we aim to investigate the effect of the projection approach \cite{gkimisis2023adjacency, prakashROM} to the properties of the resulting non-intrusive model, for the enforcement of algebraic constraints such as divergence freeness. 

\section{Conclusion \& Future Directions}
\label{sec:conc}

In this work, we investigated the method for sFOM inference in discrete-time dynamical systems. We formulated the main learning problem and established connections to past works in the field. On this basis, we derived theoretical results that provide insight into the capabilities and limitations of inference via sFOM and tested them for simple, linear dynamical systems. Moreover, we showcased the importance of the training data spatial and temporal discretization for sFOM inference via the derivation of a ``sampling CFL" condition and validated the theoretical findings with numerical results on 1D linear advection. Using these theoretical findings, we inferred sFOMs for two nonlinear systems, namely a 2D Burgers' test case and 2D, oscillating lid - driven cavity. These results indicate the potential, but also the limitations of sFOM inference in discrete-time for more realistic applications.

The current work can be extended towards the inference of parametric dynamical systems, especially in cases of parametric bifurcations, which find wide application in engineering. Such an example would be the incompressible flow around a cylinder at different inlet velocities. In parallel, an extensive study of corresponding limitations for the inference of dynamical systems via methods such as SINDy and Operator Inference would provide insight in identifying and overcoming potential failure mechanisms, towards the successful inference of more complex and large-scale dynamical systems. Finally, the established connections between regularization and stability, as well as the presented "sampling CFL" condition could provide useful insights on a more targeted selection of regularization values and strategies, for both sFOM in non-unifom grids, as well as other Machine Learning fields employing discrete maps, such as convolutional autoencoders. 

\section*{Acknowledgements}
The authors acknowledge support from the Deutsche Forschungsgemeinschaft (DFG, German Research Foundation) - 314838170, RTG 2297 MathCoRe.

\addcontentsline{toc}{section}{References}
\bibliographystyle{plainurl}
\bibliography{exampleref}
  
\end{document}